\documentclass{amsart}

\usepackage{amsmath,amsthm,indentfirst}
\usepackage{amssymb}
\usepackage{amsfonts}
\usepackage{amscd}
\usepackage[latin1]{inputenc}
\DeclareMathAlphabet{\mathpzc}{OT1}{pzc}{m}{it}
\theoremstyle{plain}

\newtheorem{thm}{Theorem}[section]
\newtheorem{cor}[thm]{Corollary}
\newtheorem{lem}[thm]{Lemma}

\newtheorem*{maintheorem*}{Main Theorem}
\newtheorem*{thm*}{Theorem}
\newtheorem*{thma*}{Theorem A}
\newtheorem*{thmaa*}{Theorem A'}
\newtheorem*{thmb*}{Theorem B}
\newtheorem*{thmo*}{Theorem 1.1}
\newtheorem*{thmc*}{Theorem C}
\newtheorem*{thmd*}{Theorem D}
\newtheorem*{thmf*}{Theorem 4.1}
\newtheorem*{remark*}{Remark}
\newtheorem*{conjecture*}{Conjecture}
\newtheorem*{prop*}{Proposition}
\newtheorem*{lem*}{Basic Lemma}

\theoremstyle{definition}

\newtheorem*{proofc*}{Proof of Theorem C}

\newtheorem{definition}[thm]{Definition}

\newtheorem{remark}[thm]{Remark}

\def\bbr{\mathbb{R}}

\def\vare{\varepsilon}

\def\d{\partial}

\def\p{\prod_{\nu\in S}}

\def\h{\hspace{1mm}}

\def\d{\partial}
\def\g{\nabla}
\def\p{\prod_{\nu\in S}}

\def\M{\mathcal{M}}

\def\p{\mathcal{P}}

\def\h{\hspace{1mm}}
\def\M{\mathcal{M}}
\def\ld{\lambda}
\def\defeq{\mathrel{\mathop:}=}
\def\sideremark#1
{
\ifvmode\leavevmode\fi\vadjust
  {\vbox to0pt
    {\vss\hbox to0pt
      {\hskip\hsize\hskip1em\vbox
        {\hsize2cm
          \scriptsize
          \raggedright\pretolerance10000\noindent#1\hfill
        }
        \hss
      }
      \vbox to8pt{\vfil}\vss
    }
  }
}

\begin{document}

\title{Elliptic equations with singular BMO coefficients in Reifenberg domains} 

\begin{abstract}
$W^{1, p}$ estimate for the solutions of elliptic equations whose coefficient matrix can have large jump along the boundary of subdomains is obtained. The principal coefficients are supposed to be in the John-Nirenberg space with small BMO seminorms. The domain and subdomains are Reifenberg flat domains and moreover, it has been shown that the estimates are uniform with respect to the distance between the subdomains. 
\end{abstract}


\author{Ko Woon Um}
\address{Mathematics Department, The University of Iowa, 14 Maclean Hall, Iowa City, Iowa 52242.}

\maketitle

\section{Introduction}

We consider the following Dirichlet problem for the divergence form elliptic equation
\begin{equation}~\label{div}
\begin{cases}-(a_{ij}u_{x_j})_{x_i} = -\mbox{div}( A (x)\g u(x)) =  \mbox{div} f  = (f^i)_{x_i} \h\h \mbox{in}\h \Omega  \\ \hspace{1.7cm} u = 0 \h\h\h\h\h\h\h\h\mbox{on} \h \d \Omega
\end{cases}\end{equation}
where $\Omega$ is an open and bounded subset of $\bbr^n$.
Throughout this paper we assume that the $n \times n$  matrix  $A=\sum_{i=0}^{i=K}A^i\chi_{\Omega^i}$ is defined on $\bbr^n$ where $ \Omega^1, \dots, \Omega^K$ are open subsets of $\Omega$ with flat boundary (see definition \ref{flatdefinition}), $  \Omega^0 \defeq {\Omega\setminus\cup_{i=1}^{i=K}\Omega^i} $ and $A^i$'s for $i=0,\dots,K$ are in the John-Nirenberg space BMO \cite{JN61} of the functions of bounded mean oscillation with small BMO seminorms.

This problem arises from the underground water flow through composite media with closely spaced interfacial boundaries, by which the coefficient matrix $A$ has discontinuity across the boundaries of subdomains. There have been many results to prove $ C^{1, \alpha} $ regularity for a weak solution in \cite{LV}, \cite{LN} and \cite{AW} by Y. Li, L. Nirenberg, M. Vogelius, F. Almgren and L. Wang. In this paper, I proved $ W^{1,p} $ regularity for Elliptic Dirichlet problem with singular coefficient matrix $A$ under some necessary conditions.

\begin{definition}\label{deltaRvanishing}(Small BMO seminorm Assumption) We say that the matrix $A$ of coefficients is $(\delta, R)-vanishing $ in $\Omega$ if
$$\sup_{0<r\leq R} \sup_{x \in \mathbb{R}^n} \sqrt{\frac{1}{|B_r|}\displaystyle \int_{B_r(x) \cap \Omega} |A(y)-\bar{A}_{B_r(x)\cap \Omega}|^2} dy \leq \delta.$$
\end{definition}

\begin{definition}\label{flatdefinition}(Reifenberg Flat Domain Assumption) We say that a domain $\Omega$ is $(\delta, R)$-{\it Reifenberg flat} if for every $x \in \d \Omega$ and every $r \in (0,R]$, there exists orthonormal coordinate system $(y_1, \dots, y_n)$ with origin at $x$ so that in that coordinate system
$$B_{r}(0) \cap \{y_n > r\delta\} \subset \Omega,$$
$$B_{r}(0) \cap \{y_n < -r\delta\} \subset \Omega^c.$$
\end{definition}

From this definition, we can see that if a domain $\Omega$ is $(\delta, R)-Reifenberg \h flat$, then for any $x \in \d \Omega$ and every $r \in (0, R]$, there exists an $(n-1)$-dimensional plane $\p(x,r)$ such that
$$ \frac{1}{r} D[\d\Omega \cap B_r(x) , \p(x,r) \cap B_r(x) ] \leq \delta,$$ where $D$ denotes the Hausdorff distance; namely,
$$D[A, B] = \sup \{{\rm{dist}}(a, B) : a \in A \} + \sup \{{\rm{dist}}(b, A) : b \in B \}.$$

We will get $W^{1,p}$ estimate for the classical weak solution of a divergence form elliptic equation (\ref{div}). The following is the definition for a weak solution.

\begin{definition}
 Let $1 < p,q<\infty$, $\frac{1}{p} + \frac{1}{q} = 1.$ Then a weak solution of (\ref{div}) is a function $u \in W_{0}^{1,p}(\Omega)$ such that
$$\displaystyle \int_{\Omega} A \g u \g \varphi dx = - \displaystyle \int_{\Omega} f \g \varphi dx \h\h\h \forall \varphi \in W_{0}^{1,q}(\Omega).$$
\end{definition}

The following is the main result of this thesis.

\begin{thm}~\label{main}Let $p$ be a real number $ 1 < p < \infty.$ Then there is a small $\delta= \delta(\Lambda, p , n, R)> 0 $ so that for all $\Omega = \cup_{i=0}^{i=K} \Omega^i$ where $\Omega^0 \defeq {\Omega\setminus\cup_{i=1}^{i=k}\Omega^i}$ and $\Omega$ and disjoint subdomains $\Omega^i$'s for $i=1,\dots,K$ are $(\delta, R)$-Reifenberg flat, for all  $A=\sum_{i=0}^{i=K}A^i\chi_{\Omega^i}$ where $A^i$'s are  $(\delta, R)$-vanishing in $ \Omega^i$ and  uniformly elliptic for $i=0,\dots,K$, and for all $f$ with $f \in L^p(\Omega, \bbr^n),$ the Dirichlet problem (\ref{div}) has a unique weak solution with the estimate
\begin{equation}~\label{result} \displaystyle \int_{\Omega}|\g u |^p dx \leq C \displaystyle \int_{\Omega}|f|^p dx,\end{equation}
where the constant $C$ is independent of $u$ and $f$.
\end{thm}

Let us just mention here that the constant $C$ above does not depend on the distance between the subdomains, which allows the domains to touch each other.

Before our work, in the parabolic case, Fred Almgren and Lihe Wang proved the  $C^{1, \alpha}$ estimates for heat flows across an interface under reasonable further assumption on $A$ in~\cite{AW}. If $u $ is a weak solution of
\begin{equation}~\label{AWsingu}
\begin{cases} B(x) u_t = \mbox{div} (A(x) \g u) + \mbox{div} F   \hspace{3mm} \mbox{in}\h\h \Omega  \\ \hspace{1.0cm} u = 0 \hspace{3.2cm}\mbox{on} \h\h \d \Omega
\end{cases}\end{equation}
where $ B(x)$ and $A(x)$ have singularity along the H\"{o}lder continuous boundaries of subdomains, they proved $|  \g u(x,s) -\g u(y,s) | \leq C |x-y|^{\alpha}$ and $ |  \g u(y,s) -\g u(y,t) | \leq C |s-t|^{\frac{\alpha}{2}}$.\\

In the elliptic case, in \cite{LV}, Y. Li and M. Vogelius considered an elliptic equation
\begin{equation} \mbox{div}(A \g u)= h + \mbox{div} (g)\end{equation}
on a bounded domain $D$ which has a finite number of disjoint subdomains $D_m$ with $ C^{1, \alpha} $ boundary and allowed the matrix $A$ to have discontinuity across the boundaries. They proved a $ C^{1, \alpha}$ regularity for the solution  under reasonable H\"{o}lder continuity assumptions on $A$, $h$ and $g_i$. Later in ~\cite{LN}, Y. Li and L. Nirenberg extended the result in \cite{LV} to general second order elliptic systems with piecewise smooth coefficients, which arises in elasticity. 


In chapter 2, we state preliminary notations, definitions and assumptions throughout this paper. Mathematical background and main tools are given in chapter 3. In the first section of chapter 4, we discuss the interior $W^{1,p}$ regularity  for a weak solution of (\ref{div}) and in the second section, a global $W^{1,p}$ regularity is derived.

\section{Definitions and Notations}
\subsection{Geometric Notation}

\begin{itemize}
\item[(1)] A typical point in $\bbr^n$ is $x = (x', x_n)$. A typical point in $\bbr^n \times \bbr$ is $(x,t) = (x', x_n, t)$.\\
\item[(2)] $\bbr^n_+ = \{ x \in \bbr^n ; x_n > 0 \}$ and $\bbr^n_- = \{ x \in \bbr^n ; x_n < 0 \}.$\\
\item[(3)] $B_r = \{ x \in \bbr^n : |x| < r \}$ is an open ball in $ \bbr^n $ centered at 0 and radius $ r >0$, $ B_r(x) = B_r + x,$  $ B_r^+ = B_r \cap \{ x_n > 0\}$, $B_r^+(x) = B_r^+ + x,$ $ T_r = B_r \cap \{ x_n = 0\}, $ and $T_r(x) = T_r + x. $\\
\item[(4)] $\Omega_r = \Omega \cap B_r, $ $ \Omega_r(x) = \Omega \cap B_r(x) $.\\
\item[(5)] $ \d \Omega_r $ is the boundary of $\Omega_r$, $ \d_w \Omega_r = \d \Omega \cap B_r $ is the wiggled part of $ \d\Omega_r, $ and $ \d_c \Omega_r = \d \Omega_r \backslash \d_w \Omega_r  $ is the curved part of $\d\Omega_r.$\\
\item[(6)] $\p_i^{\delta}(y)$ is the $(n-1)$ dimensional plane which is translated hyperplane at $y \in \d \Omega^i$ by $\delta$ along the normal direction toward $\Omega^i$.\\
\end{itemize}

\subsection{Matrix of Coefficients}
\begin{definition} We say that $A$ is $uniformly \h elliptic$ if there exists a positive constant $\Lambda $ such that
$$\Lambda^{-1} | \xi|^2 \leq A(x) \xi \cdot \xi \leq \Lambda | \xi|^2 \h\h\h\h \mbox{a.e.} \h\h\h\h x \in \bbr^n, \forall \xi \in \bbr^n. $$
\end{definition}

\begin{itemize}
\item[(1)] We write $A= (a_{ij})$ to mean an $ n \times n$ matrix with the (i, j)-th entry $a_{ij}$.
\item[(2)] $|A| = \sqrt{(A:A)} = \sqrt{\sum_{i,j=1}^{n} a_{ij}^2}$  and $\| A \|_{\infty} = \sup_{y} |A(y)|$.
\item[(3)] $$\bar{A}_{\Omega} = \frac{1}{|\Omega|} \displaystyle \int_{\Omega} A(x) dx$$ is the average of $A$ over $\Omega$.
\item[(4)] $A$ is supposed to be $A = \sum_{i = 0}^K A^i \chi{\Omega^i}$ where $A^i$'s are assumed to be uniformly elliptic and $(\delta, R)- vanishing$ on $\Omega^i$ for any $i = 0, \dots, K$.
\item[(5)] $\tilde{A}_{B_r} \defeq \sum_{i = 0}^K \bar{A^i}_{\Omega^i_r} \chi_{\Omega^i_r}$.
\end{itemize}

\subsection{Notation for Derivatives}
\begin{itemize}
\item[(1)] $\g u = (u_{x_1}, \dots, u_{x_n}) $ is the gradient of $u$.
\item[(2)]  Multiindex Notation:
\begin{itemize}
\item[(a)]A vector of the form  $\alpha= (\alpha_1, \dots, \alpha_n)$, where each component $\alpha_i$ is a nonnegative integer, is called a multiindex of order
    $$|\alpha| = \alpha_1 + \cdots + \alpha_n.$$
\item[(b)]Given a multiindex $\alpha= (\alpha_1, \dots, \alpha_n)$, define 
$$D^{\alpha} u(x) \defeq \frac{\d^{|\alpha|} u(x)}{\d x_1^{\alpha_1}\cdots \d x_n^{\alpha_n}} = \d_{x_1}^{\alpha_1} \cdots \d_{x_n}^{\alpha_n}u.$$
\end{itemize}
\item[(3)]  $$ \mbox{div}(f) = \sum_{i =1}^{n}(f^i(x))_{x_i}$$
is the divergence of $f = (f^1, \dots, f^n)$.
\end{itemize}

\subsection{Notation for estimates}
We employ the letter $C$ to denote a universal constant usually depending on the dimension, ellipticity and the geometric quantities  of $\Omega$.

\subsection{Notation for Function and Function Spaces}
\begin{itemize}
\item[(1)] If $f : \Omega \rightarrow \bbr^n$, we write $f(x) = (f^1(x), \dots, f^n(x))$ for $ x \in \Omega$.
\item[(2)] $$\bar{f}_{\Omega} = \frac{1}{|\Omega|} \displaystyle \int_{\Omega} |f(x)| dx$$ is the average of $f$ over $\Omega$.
\item[(3)] $C_{0}^{\infty}(\Omega)$ = $\{ u \in C^{\infty}(\Omega) : u \mbox{ has compact  support in }  \Omega\}.$
\item[(4)] $L^p(\Omega) = \{ u : \|u\|_{L^p(\Omega)} < \infty \},$ where $\|u \|_{L^p(\Omega)} = \left(\displaystyle \int_{\Omega}|u|^p dx \right)^{\frac{1}{p}} $ for any $1 \leq p  < \infty. $
\item[(5)] $L^{\infty}(\Omega) = \{ u : \|u\|_{L^{\infty}(\Omega)} < \infty \},$ where $\|u \|_{L^{\infty}(\Omega)} =$ ess $\sup_{\Omega }|u |.$
\item[(6)] Let $u$ and $v$ be two locally integrable functions. Then we say that $v$ is the $i^{th}$ weak derivative of $u$ if for any $ \varphi \in  C_{0}^{\infty}(\Omega),$
$$\displaystyle \int_{\Omega} u \frac{\d \varphi}{\d x_i} dx = - \displaystyle \int_{\Omega} v \varphi dx.$$
We denote by $\frac{\d u}{\d x_i} $ the $i^{th}$ weak derivative of $u$.  Then we say that $u$ is in the space $W^{1,p}(\Omega) $ if $u$ has weak derivatives $\frac{\d u }{\d x_i} \in L^p(\Omega)$ and $ u \in L^p(\Omega)$. $ W^{1,p}$ is a Banach space equiped with the norm
$$\left( \|u\|^{p}_{L^p(\Omega)} + \sum_i \|\frac{\d u }{\d x_i}\|^{p}_{L_p(\Omega)} \right)^{\frac{1}{p}}.$$ In the case $ p = 2, H^1 = W^{1,2}$ is a Hilbert space. We say $ u \in W^{1,p}_{0}(\Omega)$ if $ Eu \in W^{1,p}(\bbr^n), $ where $ Eu $ is the 0-extension of $u$ to $\bbr^n$.
\end{itemize}


\section{Preliminary tools and mathematical background}
In this chaper we recall standard facts from measure theory and functional analysis which will be needed in the sequel. We will present the proof for less familiar facts.

\subsection{The Hardy-Littlewood Maximal Function and Related Mathematical Background}
\begin{lem}\label{lplevel}\cite{CC95} Suppose that $f$ is a nonnegative measurable function in a bounded domain $\Omega$. Let $\theta > 0$ and $m > 1$ be constants. Then for  $ 0 < p < \infty,$
$$f  \in L^p(\Omega) \h\h\h\h  iff \h\h\h\h S = \sum_{k \geq 1} m^{kp} | \{ x \in \Omega : f(x) >\theta m^k\}| < \infty$$
and
$$ \frac{1}{C} S \leq \|f\|_{L^p(\Omega)}^{p} \leq C (|\Omega| + S),$$
where $C > 0$ is a constant depending only on $\theta, m $ and $p$.

\end{lem}

One of our main tools will be the Hardy-Littlewood maximal function. The maximal function controls the local behavior of a function in an analytical way.

\begin{definition}
For a locally integrable function $f$ on $\bbr^n$. Let
$$(\M f)(x) =  \sup_{r > 0} \frac{1}{|B_r(x)|} \displaystyle \int_{B_r(x)} |f(y)| dy$$
be the Hardy-Littlewood maximal function of $f$. We also define
$$\M_{\Omega} f = \M ( \chi_{\Omega} f)$$
if $f$ is not defined outside $\Omega.$
\end{definition}

\noindent
The basic theorem for the Hardy-Littlewood maximal function is the following:

\begin{thm}\label{pp11}\cite{S93} We have
\begin{itemize}
\item[(a)] If $f \in L^{p}(\bbr^n) $ with $p >1$, then $ \M f \in L^{p}(\bbr^n)$. Moreover, $$\| \M f \|_{L^{p}(\bbr^n)} \leq C\|f\|_{L^p(\bbr^n)}. $$
\item[(b)] If $f \in L^{1}(\bbr^n) $, then $$ |\{ x \in \bbr^n : (\M f) (x) > \ld \}|\leq \frac{C}{\ld}\|f\|_{L^1 (\bbr ^{n})}.$$
\end{itemize}
\end{thm}

\noindent 
Here $C$ depends only on $p$ and the dimension $n$. $(a)$ is called a strong p-p estimate and $(b)$ is called a weak 1-1 estimate. This theorem says that the measure of $ \{ x :  |\M f(x) | > \delta \} $ decays roughly as  the measure of  $ \{ x : |f(x)| > \delta \} $  does. Since the value of $L^p$ function at a particular point does not make good sense in a qualitative way even though the point is a Lebesgue point, we will employ the Hardy-Littlewood maximal function, which makes sense at a certain point. Let us also remark that the maximal funciton is invariant with respect to scaling. Hence $ |\{ x : |\M f(x) | > \delta \} | $ is more stable and geometric object.

\subsection{Vitali Covering Lemma}
Another main tool is the Vitali covering lemma:

\begin{lem}\label{vitali}\cite{S93} Let $ E $ be a measurable set. Suppose that a class of balls $ B_\alpha $ covers $E$:
$$E \subset \bigcup_{\alpha} B_{\alpha}.$$
Suppose the radius of $ B_{\alpha} $ is bounded from above. Then there exist disjoint $\{ B_{\alpha_i} \} ^{\infty}_{i = 1} \subset \{ B_{\alpha} \} _{\alpha} $ such that
$$E \subset \bigcup_{i} 5B_{\alpha_i},$$
where $5B_{\alpha_i}$ is the ball with five times the radius of $B_{\alpha_i}$ and the same center. Consequently, we have
$$|E| \leq 5^n \sum_{i} |B_{\alpha_i}|.$$
\end{lem}

For the discussion of interior $W^{1,p}$ regularity, we will use the modified version of the Vitali covering lemma:

\begin{lem}\label{insidevitali}\cite{W03} Assume that $C$ and $D$ are measurable sets, $ C \subset D \subset B_1$, and that there exists an $ \vare > 0, $ such that $$|C| < \vare|B_1|,$$
and for all $ x \in B_1$ and for all $ r \in (0, 1] $ with  $ | C \cap B_r(x) | \geq  \vare |B_r(x) |, $
$$B_r(x) \cap B_1 \subset D.$$
Then $$|C| \leq 10^n \vare |D|.$$
\end{lem}

We will use another version of the Vitali covering lemma for the global estimate on a $ (\delta, 1)-$ Reinfenberg flat domain.

\begin{lem}~\label{modifiedvitali}\cite{BW03} Assume that $C $ and $D$ are measurable sets. $C \subset D \subset \Omega$ with $\Omega$ ($\delta,$ 1)- Reifenberg flat, and that there exists an $\vare > 0$ such that
\begin{equation}~\label{clessthanb}| C | < \vare | B_1| \end{equation}
and for all $ x \in B_1$ and for all $r \in (0, 1]$ with $|C \cap B_r(x)| \geq \vare | B_r(x)|$,
\begin{equation}~\label{brind}B_r(x) \cap \Omega \subset D.\end{equation}
Then $$|C| \leq  ( \frac{10}{1-\delta} )^n \vare |D|.$$
\end{lem}

\begin{proof}From (\ref{clessthanb}), there exists a small $r_x > 0$ such that
\begin{equation}\label{rxexist} |C \cap B_{r_x}(x)| = \vare | B_{r_x}(x)|,\h\h\h\h\h |C \cap B_r(x)| \leq \vare | B_r(x)|, \h\h\h \forall r \in (r_x, 1]. \end{equation}

\noindent
Since  $\{C \cap B_{r_x}(x): x \in C\}$ is a covering of $C$ with $ r_x \leq 1$, by the Vitali covering lemma, there exists a disjoint $ \{C \cap B_{r_i}(x_i) : x_i \in C \}^{\infty}_{i = 1} $ such that

\begin{equation}~\label{riexist} C \subset  \bigcup_{i}B_{5r_i}(x_i), \h\h\h\h |C| \leq 5^n \sum_{i} |B_{r_i}(x_i)|. \end{equation}
Then, by (\ref{rxexist}), we see that
\begin{equation}~\label{riequality}| C \cap  B_{5r_i}(x_i)| < \vare |B_{5r_i}(x_i)| = 5^n \vare |B_{r_i}(x_i)| =  5^n | C\cap B_{r_i}(x_i)|. \end{equation}

Now we claim that
\begin{equation}~\label{flateffect} \sup_{0 < r \leq 1} \sup_{x \in \Omega} \frac{|B_r(x)|}{|B_r(x) \cap \Omega|} \leq (\frac{2}{1-\delta})^n. \end{equation}

\noindent
To do this, choose any $r \in (0, 1] $ and any $ x \in \Omega. $ The case dist$(x, \d \Omega) \geq r $ follows form the fact $B_r(x) \subset \Omega.$ So suppose that dist$(x, \d \Omega) < r $. Then there exists a $ y \in \d \Omega $ so that
\begin{equation}  {\rm{dist}} (x, \d \Omega) = {\rm{dist}}(x, y ) < r. \end{equation}
Since $\d \Omega$ is $(\delta, 1)$-Reifenberg flat, without loss of generality we may assume
$$B_r(x) \cap \{ x_n > \delta\} \subset B_r(x) \cap \Omega \subset B_r(x) \cap \{ x_n > -\delta\}$$
in some appropriate coordinate system in which $y = 0.$ then from the geometry and an easy computation, we see that
$$\frac{|B_r(x)|}{|B_r(x) \cap \Omega|} \leq \frac{|B_r(x)|}{|B_r(x) \cap \{ x_n > \delta \} |} \leq (\frac{2}{1-\delta})^n,$$
which shows (\ref{flateffect}). 

Finally, by (\ref{brind}), (\ref{riequality}), and (\ref{flateffect}), we get
\begin{eqnarray*}
|C| & = &| \bigcup_{i}(B_{5r_i}(x_i) \cap C ) | \\
&\leq &\sum_{i} |B_{5r_i}(x_i) \cap C  | \\
&< &\vare \sum_{i} |B_{5r_i}(x_i) | \\
& = &5^n \vare \sum_{i} |B_{r_i}(x_i) | \\
& \leq & 5^n \vare (\frac{2}{1-\delta})^n\sum_{i} |B_{r_i}(x_i)\cap \Omega | \\
& = &\vare (\frac{10}{1-\delta})^n |\bigcup_{i}(B_{r_i}(x_i)\cap \Omega )| \\
& \leq&  \vare (\frac{10}{1-\delta})^n |D|,
\end{eqnarray*}

\noindent which completes the proof.
\end{proof}

\section{Regularity for Elliptic Equations}
\subsection{Interior Estimates}
In this section we investigate the interior $W^{1,p}$ estimates for a solution of
\begin{equation}\label{intdiv}-\mbox{div}(A(x) \g u ) = \mbox{div} f \h\h\h\h \mbox{in} \h \Omega.\end{equation}
Our assumption is that $\Omega$ is bounded open set in $\bbr^n$ and  the coefficient matrix  $A=\sum_{i=0}^{i=K}A^i\chi_{\Omega^i}$ is defined on $\bbr^n$ where $ \Omega^1, \dots, \Omega^K$ are open subsets of $\Omega$ with flat boundary (see definition \ref{flatdefinition}), $  \Omega^0 \defeq {\Omega\setminus\cup_{i=1}^{i=K}\Omega^i} $ and $A^i$'s for $i=0,\dots,K$ are uniformly elliptic and also $(\delta, R)$-vanishing on $\Omega^i$ with small BMO seminorms for $i=0,\dots,K$.

$W^{1,p}$ estimate without discontinuity in $A$ was done by S. Byun and L. Wang in \cite{BW03}. Here we consider the case that $A$ has discontinuity along the boundary of subdomains $\Omega^i$'s in  $\Omega$ for $i=1,\dots,K$.

The main result of this section is the following:

\begin{thm}\label{110main}There is a constant $N_1$ so that for any $\vare > 0$,  there exists a small $\delta= \delta(\vare) > 0 $ such that for all $f \in L^2(B_{4} ; \bbr^n)$ and for all $A=\sum_{i=0}^{i=K}A^i\chi_{\Omega^i}$ where $A_i$'s are uniformly elliptic and $(\delta, 4)$-vanishing for $i = 0, \dots, K$and $\Omega^i$'s for $i = 1, \dots, K$ and $\Omega$ are $(\delta, 9)$-flat, if $u$ is a weak solution of $ -{\rm{div}}(A \g u) = {\rm{div}} f$ in $\Omega \supset B_{4}$ and if 
$$| \{x \in \Omega : \M(|\g u |^2 )(x) > N_1^2\} \cap B_r | \geq \vare |B_r| \hspace{2mm}\mbox{for all}\hspace{2mm} r \in (0, 1],$$
then
$$B_r \subset \{ x \in \Omega : \M (|\g u |^2 )(x) > 1\} \cup \{ x \in \Omega : \M (|f |^2 )(x) > \delta^2\}.$$
\end{thm}

\begin{definition}\label{weakdef} We say that $ u \in H^1(B_R)  \h (R>0) $ is a weak solution of (\ref{intdiv}) if
$$\displaystyle \int_{B_R} A \g u \g \varphi dx =  - \displaystyle \int_{B_R} f  \g \varphi dx \h\h\h\h \mbox{ for } \h\h \forall \varphi \in H_0^1(B_R).$$
\end{definition}

\begin{lem}\label{dufu}\cite{BW03}Assume that $u$ is a weak solution of (\ref{intdiv}) in $B_2$. Then
\begin{equation}\label{dufueq}\displaystyle \int_{B_2} \varphi^2 | \g u |^2 dx \leq C (\displaystyle \int_{B_2} \varphi^2 |f|^2 dx + \displaystyle \int_{B_2} |\g \varphi|^2|u|^2 dx) \h\h\h \mbox{for any} \h\h \varphi \in C_0^{\infty}(B_2).\end{equation}
\end{lem}

We want to control the gradient of the weak solution of (\ref{intdiv}) using the gradient of the weak solution of the related homogenous equation. The following lemma shows that one can bound the gradient of homogenous solution by $L^2$-norm.

\begin{lem}\label{dvbound}
If $v$ is a weak solution of ${\rm{div}}(\bar{A} \g v(x)) = 0$ in $B_1$ for a piecewise constant matrix $\bar{A} = \bar{A^1} \chi_{B_1 \cap \{x_n > a\}}+ \bar{A^0} \chi_{B_1 \cap \{x_n < a\}}$ for any $a \in (-1, 1)$,  then $$\| \g v \|_{L^\infty(B_{\frac{1}{2}})} \leq C \| v \|_{L^2(B_1)}.$$
\end{lem}

\begin{proof}
First assume $a = 0$. Let $D_i^h v(x) = \frac{v(x+ he_i)-v(x)}{h}$, for $h > 0, i = 1, \dots, n-1$.  Since the jump of the coefficient matrix $\bar{A}$ occurs across $\{ x_n = 0 \} $,
$$ \mbox{div}(\bar{A} \g D_i^h v(x)) = 0$$
for sufficiently small $h > 0.$ Also
\begin{align}
\displaystyle \int_{B_{\frac{1}{2}+ \frac{1}{4}}}|\g D_i^h v(x)|^2dx  & \leq C\displaystyle \int_{B_{\frac{1}{2}+ \frac{1}{4}+\frac{1}{8}}} |D_i^h v(x)|^2 dx \\ 
& \leq C \displaystyle \int_{B_{\frac{1}{2}+ \frac{1}{4}+\frac{1}{8}+\frac{1}{16}}} |\g v(x)|^2 dx \\
&\leq C \displaystyle \int_{B_1} | v(x)|^2 dx
\end{align}
for $0 < h <\frac{1}{16} $. Here we used the Lemma \ref{dufu} for the first and the third inequality. So $v_{x_i} \in H^1(B_{\frac{3}{4}}) $ for $  i = 1, \dots, n-1$. Similarly, we can apply this method to $v_{x_i}$, i.e. using $D_j^h v_{x_i}(x)$ for $  i, j = 1, \dots, n-1$. So $v_{x_i x_j} \in H^1(B_{\frac{1}{2} + \frac{1}{8}}) $ for $  i = 1, \dots, n-1$. Let $S = [\frac{n}{2}]+3$. For any tangential vector $\alpha = (\alpha_1, \dots, \alpha_{n-1}, 0)$ such that $|\alpha| \leq S$, we can iterate $|\alpha|$ times and get 
$$D^\alpha v(x) \in H^1(B_{\frac{1}{2} + \frac{1}{2^{S+1}}}). $$ 
Since $\mbox{div}(\bar{A} \g D^{\alpha}v(x)) = 0$, we can use the De Giorgi-Nash theorem to say that $D^{\alpha} v$ is H\"{o}lder continuous. So there is a constant $C$ such that
\begin{align}\label{alphabound}\| D^{\alpha} v \| _{L^\infty(B_{\frac{1}{2} + \frac{1}{2^{S+2}}})} &\leq C\|D^{\alpha} v \| _{L^2(B_{\frac{1}{2} + \frac{1}{2^{S+1}}})} \\
& \leq C \|v \|_{L^2(B_1)}.\end{align}

Now consider the vertical direction. 
Define $$g(x_1, x_2, \dots, x_n) \defeq v(x_1, \dots, x_{n-1}, 0) \hspace{2mm}\mbox{in} \hspace{2mm}B_{\frac{1}{2} +\frac{1}{2^{S+1}}}^+.$$
We can see that $g_{x_n} = 0$ and also by (\ref{alphabound}),
$$\begin{cases}\hspace{0.5cm}D^{\alpha} g  = D^{\alpha} v \in H^1(B_{\frac{1}{2} + \frac{1}{2^{S+1}}}^+)\\
   \hspace{0.5cm} \| D^{\alpha} g \| _{L^\infty(B_{\frac{1}{2} + \frac{1}{2^{S+2}}})}= \| D^{\alpha} v \| _{L^\infty(B_{\frac{1}{2} + \frac{1}{2^{S+2}}})} \leq C \| v \|_{L^2(B_1)}
  \end{cases}
$$
for $\alpha = (\alpha_1, \dots, \alpha_{n-1}, 0)$ such that $ |\alpha| \leq S $. 
Let $$\tilde{v}(x_1, \dots, x_n) \defeq v(x_1, \dots, x_n) - g(x_1, \dots, x_n).$$ 
Note that  $ \tilde{v} \in H^1(B_{\frac{1}{2} + \frac{1}{2^{S+1}}}^+)$ and $\tilde{v}|_{x_n= 0} = 0$. Since  $\mbox{div}(\bar{A}\g (\tilde{v} + g)) = 0 $,
\begin{align*}\mbox{div}(\bar{A}\g \tilde{v})&=-\mbox{div}(\bar{A}\g g)\\
& =- \sum_{i = 1}^{n}(\sum_{j = 1}^{n} \bar{a}_{ij}g_{x_i} )_{x_j}\\
& =- \sum_{i = 1}^{n-1}(\sum_{j = 1}^{n-1} \bar{a}_{ij}g_{x_i} )_{x_j} \in H^{S-1}(B_{\frac{1}{2}}^+)=H^{[\frac{n}{2}]+1}(B_{\frac{1}{2}}^+)
\end{align*}
Furthermore, by Theorem 5 in Section 6.3 and the Trace Theorem, see Section 5.5 in \cite{EV98}, also by Lemma \ref{dufu}, 
\begin{equation}~\label{vhsov} \|\tilde{v} \|_{H^{S-1}(B_{\frac{1}{2}}^+)} \leq C(\|v \|_{L^2(B_1)}+\|\tilde{v} \|_{L^2(B_{\frac{1}{2}})}) \leq C\| v \|_{L^2(B_1)},\end{equation}
We can combine (\ref{vhsov}) and Sobolev inequality to get 
$$\|\tilde{v} \|_{C^{S-[\frac{n}{2}]-2, \gamma}(B_{\frac{1}{2}}^+)} \leq C \|\tilde{v} \|_{H^{S-1}(B_{\frac{1}{2}}^+)} \leq C\|  v \|_{L^2(B_1)}. $$
Thus $\tilde{v}$ is $C^{1,\gamma}$ H\"{o}lder continuous. Finally we can say that $| \g \tilde{v} |$ is bounded in $ \overline{B_{\frac{1}{2}}^+}$. Similarly $| \g \tilde{v} |$ is also bounded in $\overline{B_{\frac{1}{2}}^-}$. So $| \g \tilde{v}| = | \g v - \g \tilde{g} |$ is bounded in $B_{\frac{1}{2}}$. Thus
\begin{equation}~\label{C0}\| \g v \|_{L^{\infty}(B_{\frac{1}{2}})} \leq C \| v \| _{L^2(B_1)}.\end{equation}

Assume $|a| > \frac{3}{4}$. Then $\bar{A}$ has no discontinuity in $B_{\frac{3}{4}}$. So there is a constant $C$ such that
\begin{equation}~\label{C1}\| \g v \|_{L^\infty(B_{\frac{1}{2}})} \leq C \|v \|_{L^2(B_{\frac{3}{4}})} \leq C \| v \|_{L^2(B_1)}.\end{equation}

Assume $0< |a| < \frac{3}{4} $. Say $L \defeq \{ x \in \bbr^n : x_n = a \}$.


\noindent
For any $ x \in B_{\frac{3}{4}} \cap L$, $B_{\frac{1}{4}}(x) \subset B_1$. By above case for $a = 0$, there exists a constant $C$ such that
\begin{align}~\label{C2} \| \g v \|_{L^{\infty}( \{ x \in B_{\frac{1}{2}}: {\rm{dist}}(x,L) < \frac{1}{8}\} )}  & \leq \sup_{x \in B_{\frac{3}{4}} \cap L} \| \g v \|_{L^\infty(B_{\frac{1}{8}}(x))}\\
&\leq C\|v \|_{L^2(B_{\frac{1}{4}}(x))} \leq C \| v \|_{L^2(B_1)}.\end{align}

\noindent
For any $ x \in \{ x \in B_{\frac{1}{2}}: {\rm{dist}}(x,L) \geq \frac{1}{8} \}$, $B_{\frac{1}{8}}(x) \subset B_1$ and $\bar{A}$ has no discontinuity in $B_{\frac{1}{8}}(x)$. So there exists a constant $C$ such that
\begin{equation}~\label{C3}\sup_{\{ x \in B_{\frac{1}{2}}: {\rm{dist}}(x,L) \geq \frac{1}{8} \} } \| \g v \|_{L^\infty(B_{\frac{1}{16}}(x))} \leq C \|v \|_{L^2(B_{\frac{1}{8}}(x))} \leq C \| v \|_{L^2(B_1)} .\end{equation}

\noindent
By taking the maximum $C$ in (\ref{C0}), (\ref{C1}), (\ref{C2}) and (\ref{C3}), we are done.
\end{proof}

\begin{lem}~\label{uclosev}
For any $\vare > 0,$ there is a small $ \delta= \delta(\vare) > 0$ such that for any weak solution $u$ of (\ref{intdiv}) in $B_2$  where for any $l,m = 0 \dots K$ and any $|a| < 2$ ,
\begin{equation}B_2 \cap \{ x_n > a + \delta \} \subset \Omega^l_2 \subset  B_2 \cap \{ x_n > a -\delta \} \end{equation}
\begin{equation}B_2 \cap \{ x_n < a - \delta \} \subset \Omega^m_2 \subset  B_2 \cap \{ x_n < a +\delta \} \end{equation}
and
\begin{equation}~\label{dubound}\frac{1}{|B_2|} \displaystyle \int_{B_2} | \g u |^2 dx \leq 1,\end{equation}
\begin{equation}~\label{fabound} \frac{1}{|B_2|} \displaystyle \int_{B_2} ( |f|^2 + |A - \tilde{A}_{B_2}|^2) dx \leq \delta^2,\end{equation}
where $\tilde{A}_{B_2} =  \sum_{i}\overline{A^i}_{\Omega^i_2} \chi_{\Omega^i_2},$ there exists a piecewise constant matrix $\tilde{A^b}_{B_2} $ as $ \tilde{A^b}_{B_2} = \overline{A^l}_{\Omega^l_2} \chi_{B_2 \cap \{ x_n > a\}}+ \overline{A^m}_{\Omega^m_2} \chi_{B_2 \cap \{ x_n < a\}}$ and for a corresponding weak solution $v$ of
\begin{equation}~\label{defv}-{\rm{div}}(\tilde{A^b}_{B_2} \g v) = 0 \h\h\h\h in \h\h B_2\end{equation}
such that $$\displaystyle \int_{B_2}|u-v|^2 dx \leq \vare^2.$$
\end{lem}

\begin{proof}
If not, there exists $\vare_0 > 0, \{A_k\} =\{ \sum_{i=0}^{K} A_k^i \chi_{\Omega^{i,k}}\}, \{u_k\}, \{f_k\}, \{ {\Omega^{l,k}_2}\}$ and $\{ (\Omega^{m,k})_2 \}$  for some $l,m = 0 \dots K$ and some $|a| < 2$ such that $u_k$ is a weak solution of
\begin{equation}~\label{akeq}-\mbox{div}(A_k \g u_k) = \mbox{div} f_k \h\h\h\h in \h\h\h\h B_2 \end{equation}
with
$$B_2 \cap \{ x_n >a + \frac{1}{k} \}\subset (\Omega^{l,k})_2 \subset B_2 \cap \{x_n > a -\frac{1}{k}\}$$
$$B_2 \cap \{ x_n <a - \frac{1}{k} \}\subset (\Omega^{m,k})_2 \subset B_2 \cap \{x_n < a +\frac{1}{k}\}$$ but
\begin{equation}~\label{contra1}\displaystyle \int_{B_2} |u_k - v_k |^2 dx > \vare_0^2 \end{equation}
for any weak solution $v_k$  of
\begin{equation}~\label{vkeq1} -\mbox{div}(\tilde{A_k^b}_{B_2} \g v_k) = 0 \hspace{1cm} \mbox{in} \h\h  B_2 \end{equation}
where $ \tilde{A_k^b}_{B_2} = \overline{A_k^l}_{(\Omega^{l,k})_2} \chi_{B_2 \cap \{ x_n > a\}}+ \overline{A_k^m}_{(\Omega^{m,k})_2} \chi_{B_2 \cap \{ x_n < a\}}.$

By (\ref{dubound}), $\{ u_k - \overline{u_k}_{B_2}\}_{k=1}^{\infty}$ is bounded in $H^1(B_2)$, and so  $\{ u_k - \overline{u_k}_{B_2}\}$ has a subsequence, which we denote as  $\{ u_k - \overline{u_k}\}$, such that
\begin{equation}~\label{subseq} u_k - \overline{u_k} \rightharpoonup   u_0 \h\h\h \mbox{in} \h\h\h H^1(B_2), \h\h\h
 u_k - \overline{u_k}  \rightarrow  u_0 \h\h\h \mbox{in} \h\h\h L ^2(B_2).
\end{equation}
Since $\tilde{A_k^b}_{B_2} $ is bounded in $ L^{\infty}$, there is a subsequence $\{ \tilde{A_k^b} \}$ such that
\begin{equation}~\label{aka01}\| \tilde{A_k^b}-A_0 \| _{\infty} \rightarrow 0 \h\h\h \mbox{as} \h\h\h  k \rightarrow \infty,
\end{equation}
for some piecewise constant matrix $A_0$. Since $\tilde{A_k^b}- \tilde{A_k}_{B_2} \rightarrow 0$ in $L^2(B_2)$ and $\tilde{A_k}_{B_2} - A_k \rightarrow 0$ in $L^2(B_2)$.  Thus $ A_k \rightarrow A_0$ in $L^2(B_2)$.

Next we will show that $u_0$ is a weak solution of
\begin{equation}~\label{u0eq} -\mbox{div}(A_0 \g u_0) = 0 \h\h\h \mbox{in} \h\h\h B_2\end{equation}
To do this, fix any $\varphi \in H_0^1(B_2)$. Then by (\ref{akeq}),
\begin{equation}~\label{akeq2} \displaystyle \int_{B_2} A_k \g u_k \g \varphi dx = - \displaystyle \int_{B_2} f_k \g \varphi dx. \end{equation}
Since $\g u_k \rightharpoonup  \g u_0$ and $ A_k \rightarrow A_0$ in $ L^2(B_2)$, $ A_k \g u_k \rightharpoonup A_0 \g u_0 \h\h\h \mbox{ in } \h\h\h L^2(B_2).$ Then by letting $k \rightarrow \infty$,
\begin{equation}~\label{a0u0} \displaystyle \int_{B_2} A_0 \g u_0 \g \varphi dx = 0.\end{equation}
This shows (\ref{u0eq}). Note that
\begin{align*} -\mbox{div}(\tilde{A_k^b} \g u_0 ) & = -\mbox{div}((\tilde{A_k^b}-A_0) \g u_0) -\mbox{div}(A_0 \g u_0)\\
& = -\mbox{div}((\tilde{A_k^b}-A_0) \g u_0)
\end{align*}
in $ B_2$. Let $ h_k$ be the weak solution of

\begin{equation}~\label{hkeq1}
\begin{cases}-\mbox{div}( \tilde{A_k^b} \g h_k) =  \mbox{div} ((\tilde{A_k^b}-A_0) \g u_0)  \h\h \mbox{in}\h B_2  \\ \hspace{1.8cm} h_k = 0 \h\h\h\h\h\h\h\h\mbox{on} \h \d B_2
\end{cases}\end{equation}

\noindent Then $ u_0 -h_k$ is a weak solution of
\begin{equation} \mbox{div} ((\tilde{A_k^b} \g (u_0 - h_k)) = 0  \h\h\h\h \mbox{in } B_2. \end{equation}
Furthermore, by (\ref{hkeq1}),

\begin{align*}
\|h_k\|_{L^2(B_2)} & \leq C \| \g h_k\|_{L^2(B_2)}   \leq  C \| (\tilde{A_k^b}-A_0) \g u_0\|_{L^2(B_2)} \\
& \leq C \| (\tilde{A_k^b}-A_0) \| _{L^\infty}\|\g u_0\|_{L^2(B_2)} \\
& \leq   C \| (\tilde{A_k^b}-A_0) \| _{L^\infty(B_2)}.
\end{align*}

\noindent
So now
\begin{align*}
\|u_k-(u_0+ \bar{u_k}-h_k)\|_{L^2(B_2)} & \leq  \| u_k- \bar{u_k} -u_0\|_{L^2(B_2)} + \|h_k\|_{L^2(B_2)}  \\
& \leq \| u_k- \bar{u_k}- u_0\|_{L^2(B_2)} +  C \| (\tilde{A_k^b}-A_0) \| _{L^\infty(B_2)}.
\end{align*}

\noindent
This estimate, (\ref{subseq}) and (\ref{aka01}) imply that
$$ \|u_k-(u_0+ \bar{u_k}-h_k)\|_{L^2(B_2)} \rightarrow 0 \h\h\h \mbox{as} \h\h\h k \rightarrow \infty. $$
But this is a contradiction to (\ref{contra1}) by (\ref{hkeq1}).

\end{proof}

\begin{cor}~\label{duclosedv}For any $\vare > 0,$ there is a small $ \delta= \delta(\vare) > 0$ such that for any weak solution $u$ of (\ref{intdiv}) in $B_2$  where for any $l,m = 0 \dots K$ and any $|a| < 2$ ,
\begin{equation}B_2 \cap \{ x_n > a + \delta \} \subset \Omega^l_2 \subset  B_2 \cap \{ x_n > a -\delta \} \end{equation}
\begin{equation}B_2 \cap \{ x_n < a - \delta \} \subset \Omega^m_2 \subset  B_2 \cap \{ x_n < a +\delta \} \end{equation}
and
\begin{equation}~\label{duboundc}\frac{1}{|B_2|} \displaystyle \int_{B_2} | \g u |^2 dx \leq 1,\end{equation}
\begin{equation}~\label{faboundc} \frac{1}{|B_2|} \displaystyle \int_{B_2} ( |f|^2 + |A - \tilde{A}_{B_2}|^2) dx \leq \delta^2,\end{equation}
where $\tilde{A}_{B_2} =  \sum_{i}\overline{A^i}_{\Omega^i_2} \chi_{\Omega^i_2},$ there exists a piecewise constant matrix $\tilde{A^b}_{B_2} $ as $ \tilde{A^b}_{B_2} = \overline{A^l}_{\Omega^l_2} \chi_{B_2 \cap \{ x_n > a\}}+ \overline{A^m}_{\Omega^m_2} \chi_{B_2 \cap \{ x_n < a\}}$ and for a corresponding weak solution $v$ of
\begin{equation}~\label{defv1}-{\rm{div}}(\tilde{A^b}_{B_2} \g v) = 0 \h\h\h\h in \h\h B_2\end{equation}
such that $$\displaystyle \int_{B_{\frac{4}{3}}}|\g (u- v)|^2 dx \leq \vare^2.$$
\end{cor}

\begin{proof}
By the Lemma \ref{uclosev},  for any $\eta > 0,$ there exists $ \delta= \delta(\eta) > 0,$ a piecewise constant matrix $  \tilde{A^b}_{B_2} = \overline{A^l}_{\Omega^l_2} \chi_{B_2 \cap \{ x_n > a\}}+ \overline{A^m}_{\Omega^m_2} \chi_{B_2 \cap \{ x_n < a\}}$ and a corresponding weak solution $v$ of $-\mbox{div}(\tilde{A^b}_{B_2} \g v) = 0 $ in $B_2$ such that  $$\displaystyle \int_{B_2}|u-v|^2 dx \leq \eta^2.$$

First we see that $u-v \in H^1(B_2)$ is a weak solution of
\begin{equation}~\label{u-veq} -\mbox{div}(A \g (u-v)) = \mbox{div}(f + (A- \tilde{A^b}_{B_2})\g v) \h\h\h\h\h\h\h \mbox{in} \h\h\h B_2\end{equation}
Now, by (\ref{dufueq}),
\begin{align} \displaystyle \int_{B_{\frac{4}{3}}}|\g ( u-v)|^2 & \leq C(\displaystyle \int_{B_{\frac{3}{2}}}| f + (A- \tilde{A^b}_{B_2})\g v |^2 + | u-v | ^2 dx) \\
& \leq C(\displaystyle \int_{B_{\frac{3}{2}}}|f|^2 dx + \displaystyle \int_{B_{\frac{3}{2}}}|(A-\tilde{A^b}_{B_2})\g v |^2 dx + \displaystyle \int_{B_{\frac{3}{2}}}|u-v|^2 dx) \\
& \leq C(\displaystyle \int_{B_2}|f|^2 + \displaystyle \int_{B_2}|A-\tilde{A^b}_{B_2}|^2 dx + \displaystyle \int_{B_2}|u-v|^2 dx)
\end{align}
Here we used the fact that $v$ is lipschitz, which we showed in Lemma \ref{dvbound}, and (\ref{duboundc}). Also,
\begin{align} \displaystyle \int_{B_2} |f|^2 +| A- \tilde{A^b}_{B_2}|^2 dx  & \leq 2\displaystyle \int_{B_2} (|f|^2 +| A- \tilde{A}_{B_2}|^2) +  | \tilde{A}_{B_2}- \tilde{A^b}_{B_2}|^2\\
&  \leq 2(|B_2|\delta^2 + C(\Lambda) \delta ) \\
& \leq C\delta\h\h\h  \mbox{for a small } \h \delta .
\end{align}
So  $ \| \g ( u-v) \|_{L^2(B_2)}^2 \leq C( \delta + \eta^2) = \vare^2$ by taking $ \eta$ and $\delta$ satisfying the last identity. This completes our proof.
\end{proof}

We can control the measure of the set where $| \g u |$ is quite big as the following lemma.

\begin{lem}\label{bigducontrol1} (cf.~\cite{BW03}) There is a constant $N_1 > 0$ so that for any $\vare > 0$, there exists a small $\delta = \delta(\vare) > 0$ such that for all $A$ with $ A = \sum_{i=0}^{K}A^i \chi_{\Omega^i}$, where $A^i$'s are uniformly elliptic and $(\delta, 4)$-vanishing on $\Omega^i$ for $i = 0 \dots K$ and for any $l,m = 0 \dots K$ and any $|a| < 4$ in appropriate coordinate system
\begin{equation}B_4 \cap \{ x_n > a + \delta \} \subset \Omega^l_4 \subset  B_4 \cap \{ x_n > a -\delta \} \end{equation}
\begin{equation}B_4 \cap \{ x_n < a - \delta \} \subset \Omega^m_4 \subset  B_4 \cap \{ x_n < a +\delta \}, \end{equation} and if $ u $ is a weak solution of $-{\rm{div}}(A \g u ) = {\rm{div}} f $ in $\Omega \supset B_4$ and if
\begin{equation}~\label{noempty1}\{ x \in B_1 : \M(|\g u |^2 ) \leq 1 \} \cap \{ x \in B_1 : \M (|f|^2)\leq \delta^2 \} \neq \emptyset, \end{equation}
then
\begin{equation}~\label{measuresmall1}| \{ x \in \Omega : \M (| \g u |^2)(x) > N_1^2\} \cap B_1| < \vare|B_1|.\end{equation}
\end{lem}

\begin{proof}
By (\ref{noempty1}), there is a point $x_0 \in B_1$ such that for all $r > 0$,
\begin{equation}~\label{du1fsmall}\frac{1}{|B_r|} \displaystyle \int_{B_r(x_0) \cap \Omega} | \g u|^2 dx \leq 1, \h\h\h  \frac{1}{|B_r|} \displaystyle \int_{B_r(x_0) \cap \Omega} |f|^2 dx \leq \delta^2.\end{equation}
Since $B_2(0) \subset B_3(x_0) $, we have by (\ref{du1fsmall}),
\begin{equation}~\label{fsmall}\frac{1}{|B_2|} \displaystyle \int_{B_2} |f|^2 dx \leq \frac{|B_3|}{|B_2|} \frac{1}{|B_3|} \displaystyle \int_{B_3(x_0)} |f|^2 dx \leq (\frac{3}{2})^n \delta^2.\end{equation}
Similarly, we see that
\begin{equation}~\label{dusmall} \frac{1}{|B_2|} \displaystyle \int_{B_2} |\g u |^2 dx \leq (\frac{3}{2})^n. \end{equation}
In view of (\ref{fsmall}) and (\ref{dusmall}), and from the assumption on $A$, we can apply Corollary \ref{duclosedv} with $u$ replaced by $ (\frac{2}{3})^n u $ and $f$ replaced by $(\frac{2}{3})^nf$, respectively, to find that for any $ \eta > 0$, there exists a small $ \delta(\eta) $ and a corresponding weak solution $v$ of
\begin{equation}-\mbox{div}(\tilde{A^b}_{B_2} \g v) = 0\end{equation}
in $ B_2$ such that
\begin{equation}~\label{du-vsmall}\displaystyle \int_{B_{\frac{4}{3}}} | \g ( u-v)|^2 dx \leq \eta^2,\end{equation}
provided that
\begin{equation}~\label{fa-asmall}\frac{1}{|B_2|} \displaystyle \int_{B_2} ( |f|^2 + |A - \tilde{A}_{B_2}|^2) dx \leq \delta^2.\end{equation}
By the interior $W^{1, \infty}$ regularity that we proved in Lemma \ref{dvbound}, we can find a constant $ N_0$ such that
\begin{equation}~\label{dv'slinftysmall}\| \g v \|_{L^{\infty}(B_{\frac{3}{2}})} \leq N_0.\end{equation}

Now we will show that
\begin{equation}~\label{claim1}\{ x \in B_1 : \M|\g u |^2 > N_1^2\} \subset \{ x \in B_1 : \M_{B_2}|\g ( u-v)|^2 > N_0^2\}\end{equation}
for $ N_1^2 \defeq \mbox{max}\{ 5^n, 4N_0^2\}.$ To do this, suppose that
\begin{equation}~\label{x1N0} x_1 \in \{ x \in B_1 : \M_{B_2}(|\g(u-v)|)^2(x) \leq N_0^2\}. \end{equation}
For $ r \leq \frac{1}{2} , $ $ B_r(x_1) \subset B_{\frac{3}{2}}, $ and by (\ref{dv'slinftysmall}) and (\ref{x1N0}), we have
\begin{equation}~\label{rlessthan2}\frac{1}{|B_r|}\displaystyle \int_{B_r(x_1)} |\g u |^2 dx \leq \frac{2}{|B_r|}\displaystyle \int_{ B_{\frac{3}{2}}} (|\g(u-v)|^2 + |\g v |^2) \leq 4N_0^2.\end{equation}
For $ r > \frac{1}{2}$, $ B_r(x_1) \subset B_{5r}(x_0),$  and by (\ref{du1fsmall}), we have
\begin{equation}~\label{rbiggerthan2} \frac{1}{|B_r|}\displaystyle \int_{B_r(x_1)} |\g u |^2 dx \leq \frac{5^n}{|B_{5r}|}\displaystyle \int_{B_{5r}(x_0)\cap \Omega} |\g u |^2 dx \leq 5^n. \end{equation}
Then (\ref{rlessthan2}) and (\ref{rbiggerthan2}) show
\begin{equation}~\label{x1N1} x_1 \in \{ x \in B_1 : \M(|\g u |)^2 \leq N_1^2\}.\end{equation}
Thus assertion (\ref{claim1}) follows from (\ref{x1N0}) and (\ref{x1N1}). 

By (\ref{claim1}), weak 1-1 estimates and (\ref{du-vsmall}), we obtain
\begin{align*}
|\{ x \in B_1 : \M(|\g u |)^2 > N_1^2\}| & \leq |\{ x \in B_1 : \M_{B_2}(|\g(u-v)|)^2 > N_0^2\}| \\
&\leq \frac{C}{N_0^2} \displaystyle \int_{B_{\frac{4}{3}}} |\g(u-v)|^2 dx \\
&\leq \frac{C}{N_0^2}\eta^2 = \vare |B_1|,
\end{align*}
by taking small $ \eta$ satisfying the last identity above. Now Corollary \ref{duclosedv} gives the desired $\delta$.
\end{proof}

\begin{cor}~\label{anyscaleball} There is a constant $ N_1 >0 $ so that for any $\vare, r \in (0,1]$, there exists a small $ \delta= \delta(\vare) > 0$ such that for all $A$ with $ A = \sum_{i=0}^{K}A^i \chi_{\Omega^i}$, where $A^i$'s are uniformly elliptic and $(\delta, 4)$-vanishing on $\Omega^i$ for $i = 0 \dots K$ and for any $l,m = 0 \dots K$ and any $|a| < 4r$ in appropriate coordinate system
\begin{equation}B_{4r} \cap \{ x_n > a + \delta r \} \subset \Omega^l_{4r}\subset  B_{4r} \cap \{ x_n > a -\delta r\} \end{equation}
\begin{equation}B_{4r} \cap \{ x_n < a - \delta r\} \subset \Omega^m_{4r}\subset  B_{4r} \cap \{ x_n < a +\delta r\} \end{equation} and if $ u $ is a weak solution of $-{\rm{div}}(A \g u ) = {\rm{div}} f $ in $\Omega \supset B_{4r}$ and if
\begin{equation}~\label{noemptyr} \{ x \in B_r : \M(|\g u |^2 ) \leq 1 \} \cap \{x \in B_r: \M (|f|^2)\leq \delta^2 \} \neq \emptyset, \end{equation}
then
\begin{equation}~\label{measuresmallr}| \{ x \in \Omega : \M (| \g u |^2)(x) > N_1^2\} \cap B_r| < \vare|B_r|.\end{equation}
\end{cor}

\begin{proof}
The proof is given by Lemma \ref{bigducontrol1} and a scaling argument.
\end{proof}



To use the modified vitali covering lemma, we need to show Theorem \ref{110main} holds for any ball $B_r(x)$ for $r \in (0, 1]$ and $ x \in \Omega$. If $B_r(x)$ intersects with only one subdomain $\Omega^l$ then the proof of Theorem \ref{110main} comes directly from Lemma \ref{anyscaleball} for $l=m$.
If $B_r(x)$ intersects with two subdomains $\Omega^l$ and $\Omega^0$, then the proof of Theorem \ref{110main} also comes directly from Lemma \ref{anyscaleball} for $m=0$.

Then next natural question would be how many subdomains can intersect with $B_r(x)$ for $r \in (0, 1]$ and $ x \in \Omega$ when $\d \Omega^i$'s are flat enough. Next lemma will be used to show that a ball can intersect with at most three subdomains.

\begin{lem}\label{notmore3} $H_i$'s for $i = 1, \dots, K$ are half spaces. If $\{ H_i \cap B_2\}_i$ are disjoint. Then at most two half spaces can intersect with $B_1$.
\end{lem}

\begin{proof}
Assume there are three half spaces, say $H_1, H_2$ and $H_3$ such that $B_2 \cap H_i$'s are disjoint and $H_i \cap B_1 \neq \emptyset$ for $i = 1,2,3.$
Let $ p_i \in H_i \cap B_1 $ for $i = 1,2,3.$. Note that since half spaces are disjoint in $B_2$ these points are not collinear. Let $\mathcal{T}$ be the two dimensional plane containing $p_1, p_2, p_3$. For $j=1,2$ let $\mathcal{D}_j=\mathcal{T}\cap B_j$ which are indeed two dimensional balls. Let $r_j=\mbox{radius of} \h \mathcal{D}_j$ for $j=1,2.$ Note that $r_2 \geq 2r_1.$

Let $h_i \defeq \mathcal{T} \cap  H_i $ and $l_i \defeq \mathcal{T} \cap \d H_i=\d h_i .$ We have
\begin{itemize}
\item[(1)]$ p_i\in l_i \cap \mathcal{D}_1 $ for $i=1,2,3$
\item[(2)]$h_i \cap \mathcal{D}_2$'s are disjoint for $ i = 1,2,3.$
\end{itemize}


\noindent
Pushing $l_i$'s into $h_i$ by $\delta_i>0$, we may assume that $l_i$'s are tangent to the $\mathcal{D}_1$ and $p_i \in \d \mathcal{D}_1 $ for $i = 1,2,3.$ Let also $A_i$ and $B_i$ be the points where $l_i$ intersects $\d\mathcal{D}_2$ for $i = 1,2,3.$ Let $h_i \cap \d D_2 = \stackrel \frown{A_iB_i}$.


Note that $\stackrel \frown{A_iB_i}$ for $i = 1,2,3.$ are disjoint on $\d D_2$. Since $r_2 \geq 2r_1$ and $l_i$'s are tangent to $D_1$,
\begin{equation}~\label{1/3}\frac{\mbox{length of }\stackrel \frown{A_iB_i}}{\mbox{length of }\d D_2} \geq \frac{1}{3}, \hspace{1.5cm} \mbox{for} \h\h i = 1,2,3.\end{equation}
The above is a strict inequality if $ r_2 > 2r_1$, which is a contradiction to the fact that $\stackrel \frown{A_iB_i}$'s are disjoint on $\d D_2$. If $r_2 = 2r_1$, (\ref{1/3}) is an equality. In this case $l_i$'s end points meet each other. So we cannot push $l_i$ outward from $h_i$ which means $\delta_i = 0$ for $i = 1,2,3.$
\end{proof}

So now we consider the case that a ball intersect with three subdomains $\Omega^l$, $\Omega^0$ and $\Omega^m$ for any $l,m = 1 \dots K.$ To prove Theorem \ref{110main} for this case, our goal is to show Lemma \ref{bigducontrol1} holds for this case as well. Roughly there can be two different cases; The first case is when  $\Omega^l$ and  $\Omega^m$ are quite close and the second case is when $\Omega^l$ and  $\Omega^m$ are not so close.

\begin{lem}~\label{bigducontr} There exists a constant $N_1 > 0$ so that for any $\vare > 0$, there exists a small $\delta = \delta(\vare) > 0$ and for all $\Omega \supset  B_{4}$ and subdomain $\Omega^i$ for all $i = 1, \dots, K$ and $\Omega$ are $(\delta, 9)$-flat and for all $A$ where $A^i$'s are uniformly elliptic and $(\delta, 9)$ vanishing on $\Omega^i$, and if $ u $ is a weak solution of $-{\rm{div}}(A \g u ) = {\rm{div}} f $ in $\Omega \supset  B_{4}$ and if

\begin{equation}~\label{noempty3}\{ x \in B_1 : \M(|\g u |^2 ) \leq 1 \} \cap \{x \in B_1 : \M (|f|^2)\leq \delta^2 \} \neq \emptyset,\end{equation}
then
\begin{equation}~\label{measuresmall3}| \{ x \in \Omega  : \M (| \g u |^2)(x) > N_1^2\} \cap B_1| < \vare|B_1|.\end{equation}
\end{lem}

\begin{proof}
If $B_{4}$ intersects with two subdomains, then we are done by Lemma \ref{bigducontrol1}.

Suppose $B_{4}$ intersects with three subdomains, say $\Omega^l, \Omega^0$ and $\Omega^m$. First assume that dist$(\Omega^l, \Omega^m) < \gamma $ in $B_1$ for some small $ \gamma > 0 $. Since dist$(\Omega^l, \Omega^m) < \gamma $ in $ B_1$, there exist $p_l \in \d \Omega^l \cap B_1$ and $p_m \in \d \Omega^m \cap B_1$ such that dist$(p_l, p_m) < \gamma$.  Also assume that $ \Omega^l, \Omega^m $ are $(\delta, 9) $-Reifenberg flat for a $\delta$ with $ \gamma < \delta << 1 $. So for each $p_i$, $i = l,m$, there exist $(n-1)$ dimensional hyper plane $\p_i$ such that
\begin{equation}~\label{4delta}D[\d \Omega^i \cap B_9(p_i), \p_i \cap B_9(p_i)] \leq 9\delta, \h\h\h \mbox{ for } i = l,m
\end{equation}
where $D$ denotes the Hausdorff distance. In other words, the boundary of $\Omega^i$ is squeezed between $\p_i$ and $\p_i^{9\delta}$ which is the translation of $\p_i$ by $9\delta$ in the normal direction of $\p_i$ inward $\Omega^i$ for $i = l, m$. We can choose a coordinate system such that the normal direction of $\p_l^{9\delta}$ is the $x_n$ axis. Let us say $y_i$ is the intersection point between $ \p_i^{9\delta}$ and vertical line of $\p_i^{9\delta}$ passing through $p_i$ for $i = l, m$. Then the distance between $y_m$ and $ \p_l^{9\delta}$ is less than $\gamma + 18\delta < 19\delta$ by (\ref{4delta}). Since $\p_l^{9\delta} \cap \p_m^{9\delta} \cap B_4 = \emptyset$, on $ \p_m^{9\delta}$
$$|\frac{\d x_n}{\d x_i}| < \frac{\gamma+ 18\delta}{3-\gamma-18\delta} < \frac{19\delta}{3-19\delta} <7\delta \h\h\h\h \mbox{for any } \gamma < \delta << 1, \mbox{ and } i = 1, \dots n-1. $$
So $ \max_{ y \in \p_m^{9\delta} \cap B_4} {\rm{dist}}(y, \p_l^{9\delta} \cap B_4) < C \delta + \gamma $ where $C$ depends on the dimension $n$.

\noindent
The above is nothing but harnack inequality. Since distance function between $\p_l^{9\delta}$ and $\p_m^{9\delta}$ in $B_4$  is nonnegative harmonic, we can apply Harnack Inequality.
\begin{equation}~\label{harnack} \max_{ y \in \p_m^{9\delta} \cap B_1} {\rm{dist}}(\p_l^{9\delta}, y)  < C_1 \min_{ y \in \p_m^{9\delta} \cap B_1} {\rm{dist}}(\p_l^{9\delta}, y) < C{\rm{dist}}(y_l, y_m)= C( 19\delta + \gamma)
\end{equation}
where $C$ depends on the dimension $n$.

Since the Hausdorff distance between $\p_l^{9\delta}, \p_m^{9\delta} $ is less than $C(\delta + \gamma)$, we can choose small $\delta_0$ and $\gamma_0$ such that $C(\delta_0 + \gamma_0)$ is less than $\delta$ in Lemma \ref{bigducontrol1}. By Lemma \ref{bigducontrol1}, we can conclude.

Now suppose $ {\rm{dist}} (\d \Omega^l, \d \Omega^m) > \gamma_0$ in $ B_1$ for above $\gamma_0$.
If $ y \in S_1 = \{ x \in B_1 |\h x \in \d \Omega^l \cap \d \Omega^m \h \} $, then $B_{\gamma_0}(y)$ has only two subdomains. From (\ref{noempty3}), there exists $x_0 \in B_1$ such that
$$\M(|\g u |^2 )(x_0) \leq 1 \h\h\h \mbox{ and}\h\h\h \M (|f|^2)(x_0)\leq \delta^2.$$

\noindent
For any $ y \in S_1$, by weak 1-1 estimate in Theorem \ref{pp11},

\begin{align*} |\{ x \in B_{\frac{\gamma_0}{4}}(y)  :   \M (| \g u |^2)(x) > \lambda_1\}  | & \leq \frac{C}{\lambda_1} \displaystyle \int_{B_2(x_0)} | \g u |^2 dx\\
& \leq \frac{C}{\lambda_1} |B_2(x_0)| < \frac{1}{2} |B_{\frac{\gamma_0}{4}}(y)|
\end{align*}

\noindent
when $\lambda_1 > \frac{C2^{3n+1}}{\gamma_0^n}$. Similarly for this $\lambda_1$,

\begin{align*} |\{ x \in B_{\frac{\gamma_0}{4}}(y)  :   \M (| f |^2)(x) > \delta^2 \lambda_1\}  | & \leq \frac{C}{\delta^2 \lambda_1} \displaystyle \int_{B_2(x_0)} | f|^2 dx\\
& \leq \frac{C}{\delta^2 \lambda_1} |B_2(x_0)| < \frac{1}{2} |B_{\frac{\gamma_0}{4}}(y)|.
\end{align*}

\noindent
From above two inequalities, one can find a $x_y \in B_{\frac{\gamma_0}{4}}(y)$ such that
$$\M(|\g u |^2 )(x_y) \leq \lambda_1 \h\h\h \mbox{ and}\h\h\h \M (|f|^2)(x_y)\leq \delta^2 \lambda_1.$$

\noindent
By Lemma \ref{anyscaleball}, there is a constant $N_1$ so that for any $\vare > 0$

\begin{equation}~\label{bad1}|\{ x \in \Omega  : \M (| \g u |^2)(x) > \lambda_1 N_1^2\} \cap B_{ \frac{\gamma_0}{4}}(y) | < \vare| B_{\frac{\gamma_0}{4}}(y)|.\end{equation}

\noindent
If $ y \in S_2 = \{ x \in B_1 |\h \min_{i = l, m} \h {\rm{dist}}(x, \d \Omega^i) > \frac{\gamma_0}{4 \times 5} \h \}$, $ B_{\frac{\gamma_0}{20}}(y) \subset \Omega^i$ for $ i = 0, l , m $. Similarly as above, there is a $x_y \in B_{\frac{\gamma_0}{80}}(y)$ such that
$$\M(|\g u |^2 )(x_y) \leq \lambda_2 \h\h\h \mbox{ and}\h\h\h \M (|f|^2)(x_y)\leq \delta^2 \lambda_2.$$

\noindent
when $\lambda_2 >  \frac{C2^{5n+1}5^n}{\gamma_0^n}$. By Lemma \ref{anyscaleball}, there is a constant $N_1$ so that for any $\vare > 0$

\begin{equation}~\label{bad2}|\{ x \in \Omega  : \M (| \g u |^2)(x) > \lambda_2 N_1^2\} \cap B_{ \frac{\gamma_0}{80}}(y) | < \vare| B_{\frac{\gamma_0}{80}}(y)|.\end{equation}

\noindent
So $ U = \{ B_r(y) | \h r = \frac{\gamma_0}{4 \times 5}, y \in S_1 \h \} \cup  \{ B_r(y) | \h r = \frac{\gamma_0}{80 \times 5}, y \in S_2 \h \} $ covers $ B_1 $. Then by Vitali Covering Lemma, there exist disjoint balls $\{ B_{r_i}(y_i) \}_{i = 1}^{\infty} \subset U \subset B_2$ such that $B_1 \subset  \cup_{i}B_{5r_i}(y_i)$. Let $N_1$ to be $\max (\sqrt{\lambda_1} N_1,\sqrt{\lambda_2} N_1)$. Then by (\ref{bad1}) and (\ref{bad2}),

\begin{align*} |\{ x \in \Omega  : \h & \M (| \g u |^2)(x) > N_1^2\} \cap B_1 | \\
& < \sum_{i}|\{ x \in \Omega  : \M (| \g u |^2)(x) > N_1^2\} \cap B_{5r_i}(y_i) |\\
& < \vare \sum_{i}  | B_{5r_i}(y_i)| < \vare 5^n\sum_{i}  | B_{r_i}(y_i)| \\
& < \vare 5^n |B_2| < \vare (10)^n |B_1|.
\end{align*}

Since $\Omega^i$'s for $i = 0, \dots, n$ are $(\delta, 9)$-flat, $B_{4}$ does not intersect more than three subdomains. To see that, assume that $B_{4}$ intersects with $\Omega^0, \Omega^1, \Omega^2, \Omega^3$. For any $p_i \in \d \Omega^i \cap B_{4}$, for $i = 1, 2, 3$, there exists a hyperplane $\p_i$ such that $\d \Omega^i \cap B_{9}$ is between $\p_i$ and $\p_i^{9\delta}$ where $\p_i^{9\delta}$ is translation of $\p_i$ into $\Omega^i$ in the normal direction by $9\delta$ since $\Omega^i$'s for $i = 0, \dots, n$ are $(\delta, 9)$-flat. Then for any $\delta < \frac{1}{18}$, on the plane $\mathcal{T}$ containing $p_1, p_2, p_3$, $H_i$ for $i = 1,2,3$ intersect with $B_{\frac{9}{2}}$ but they are disjoint in $B_{9}$, which is a contradiction to Lemma \ref{notmore3}.

\end{proof}

{\bf The proof of Theorem \ref{110main}} The proof follows from Lemma \ref{bigducontr} and scaling argument.

The following is an interior regularity theorem.

\begin{thm}~\label{maintheo}
 Let $p$ be a real number with $ 1 < p < \infty$. There is a small $ \delta = \delta(\lambda, p, n, R)$ so that for all $\Omega = \cup_{i=0}^{K} \Omega^i$ where $ \Omega^0 \defeq {\Omega\setminus\cup_{i=1}^K\Omega^i} $ and  $\Omega^i$'s for $i = 1, \dots, K$ and $\Omega$ are $(\delta, 9)$-flat and $A=\sum_{i=0}^{i=K}A^i\chi_{\Omega^i}$ where $A_i$'s are uniformly elliptic and $(\delta, 9)$-vanishing on $\Omega^i$ and for all $ f \in L^p(B_{4}; \bbr^n)$ , if $u$ is a weak solution of the elliptic PDE (\ref{div}) in $B_{4}$, then $u$ belong to $W^{1, p}(B_1)$ with the estimate
$$ \| \g u \|_{L^p(B_1)} \leq C(\|u \|_{L^p(B_{4})} + \|f \|_{L^p(B_{4})}),$$
where the constant $C$ is independent of $u$ and $f$.
\end{thm}

\begin{proof}
The proof follows from the global regularity theory in the next section with u replaced by $\phi u$ for an appropriately chosen cutoff function $\phi$.
\end{proof}

\begin{remark}We can change the ball $B_4$ in Theorem \ref{maintheo} to any ball $B_R$ for $R > 1.$
\end{remark}


\subsection{Global Estimates}

\begin{definition} We say that $ u \in H^1_0(\Omega) $ is a weak solution of (\ref{div}) if
\begin{equation} -\displaystyle \int_{\Omega} A \g u \g \varphi dx = \displaystyle \int_{\Omega} f \g \varphi dx \h\h\h\h \forall \varphi \in H^1_0(\Omega).
\end{equation}
\end{definition}

In this section our interest is the following case.
$$\Omega_R \supset T_R \hspace{1cm} \mbox{ with } D(\Omega_R, T_R) \h\h \mbox{small,}$$
where $D$ denotes the Hausdorff distance. We consider weak solution of

\begin{equation}~\label{divglou}
\begin{cases} -\mbox{div}( A (x)\g u(x)) =  \mbox{div} f  \h\h \mbox{in}\h \Omega_R  \\ 
\h\h\h\h\h\h u = 0 \h\h\h\h\h\h\h\h\mbox{on} \h \d_w \Omega_R
\end{cases}\end{equation}

Here the $n \times n$ coefficient matrix $A$ is $A=\sum_{i=0}^{i=K}A^i\chi_{\Omega^i}$ where $\Omega^0 \defeq {\Omega\setminus\cup_{i=1}^{i=k}\Omega^i} $ and $A^i$'s for $i=0,\cdots,K$ are in the John-Nirenberg space BMO \cite{JN61} of the functions of bounded mean oscillation with small BMO seminorms and $\Omega$ and $\Omega^i$'s are Reifenberg flat domains for $i = 1 \dots K$.




\begin{definition}  $u \in H^1(\Omega_R)$ is a weak solution of (\ref{divglou}) in $\Omega_R$ if
$$\displaystyle \int_{\Omega_R} A \g u \g \varphi dx =  - \displaystyle \int_{\Omega_R} f  \g \varphi dx \h\h\h\h \mbox{for any } \h\h \varphi \in H_0^1(\Omega_R)$$
and $u$'s $0$-extension is in $H^1(B_R).$


\end{definition}
In (\cite{BW03}), the following Lemmas were proven for $A$ without discontinuity.

\begin{lem}\label{glolem}\cite{BW03} There is a constant $ N_1 >0 $ so that for any $\vare> 0$, there exists a small $ \delta= \delta(\vare) > 0$ with $A$ uniformly elliptic and $(\delta, 4)$-vanishing, and if $ u  \in H^1_0(\Omega)$ is a weak solution of (\ref{divglou}) with $B_{4}^+ \subset  \Omega_4 \subset B_{4} \cap \{ x_n > -\delta \} $ and
\begin{equation}~\label{glnoemptyr}  \{ x \in \Omega_1 : \M(|\g u |^2 ) \leq 1 \} \cap \{x \in \Omega_1: \M (|f|^2)\leq \delta^2 \} \neq \emptyset, \end{equation}
then
\begin{equation}~\label{glmeasuresmallr}| \{ x \in \Omega : \M (| \g u |^2)(x) > N_1^2\} \cap B_1| < \vare|B_1|.\end{equation}
\end{lem}

\begin{cor}\label{glocor}\cite{BW03} There is a constant $ N_1 >0 $ so that for any $\vare, r > 0$, there exists a small $ \delta= \delta(\vare) > 0$ with $A$ uniformly elliptic and $(\delta, 4r)$-vanishing, and if $ u  \in H^1_0(\Omega)$ is a weak solution of (\ref{divglou}) with $B_{4r}^+ \subset  \Omega_{4r} \subset B_{4r} \cap \{ x_n > -\delta r \} $ and
\begin{equation} \{ x \in \Omega_r : \M(|\g u |^2 ) \leq 1 \} \cap \{x \in \Omega_r: \M (|f|^2)\leq \delta^2 \} \neq \emptyset, \end{equation}
then
\begin{equation}| \{ x \in \Omega : \M (| \g u |^2)(x) > N_1^2\} \cap B_r| < \vare|B_r|.\end{equation}
\end{cor}

Now we consider how to control the measure of the set where $|\g u|$ is big for the case that $A$ has big discontinuity along the subdomains.

\begin{lem}~\label{myglolem}There is a constant $ N_1 >0 $ so that for any $\vare> 0$, there exists a small $ \delta= \delta(\vare) > 0$ with $A_i$'s are uniformly elliptic and $(\delta, 9)$-vanishing on $\Omega^i$ for $i = 0 \dots K$ and $\Omega$ and $\Omega^i$'s are $(\delta,9)$-flat for $i = 1\dots K$, and if $ u  \in H^1_0(\Omega)$ is a weak solution of (\ref{divglou}) with $B_{4}^+ \subset  \Omega_{4} \subset B_{4} \cap \{ x_n > -4\delta \} $ and
\begin{equation}~\label{mynoemptyr}  \{ x \in \Omega_1 : \M(|\g u |^2 ) \leq 1 \} \cap \{x \in \Omega_1: \M (|f|^2)\leq \delta^2 \} \neq \emptyset, \end{equation}
then
\begin{equation}~\label{mymeasuresmallr}| \{ x \in \Omega : \M (| \g u |^2)(x) > N_1^2\} \cap B_1| < \vare|B_1|.\end{equation}
\end{lem}

\begin{proof}
If $B_4$ intersects with only $\Omega^0$, then this lemma is nothing but what Lemma \ref{glolem} says. Note that $B_{4}$ cannot intersect with more than two subdomains by the same argument in the proof of Lemma \ref{bigducontr}. (considering $\Omega^c$ as $(\delta, 9)$-flat for any sufficiently small $\delta$). Assume that $B_4$ intersects with $\Omega^0$ and $\Omega^l$ for any $l = 1 \dots K$.

First suppose dist$(\d \Omega^l, \d \Omega) < \gamma$ in $B_4$ for some $\gamma>0$. Then there exist $p_l \in \d\Omega^l \cap B_4 $ and $p \in \d\Omega \cap B_4 $ such that ${\rm{dist}}(p, p_l) < \gamma$. Since $\Omega^l$ are $(\delta, 9)$-flat, $\p_l^{9\delta}(p_l) \cap B_4\subset \Omega^l$ where $\p_l^{\delta}(p_l)$ is the $(n-1)$ dimensional plane which is translated hyperplane at $p_l$ by $\delta$ along the normal direction toward $\Omega^l$. Let us say $y_l$ is the intersection point between $ \p_l^{9\delta}$ and vertical line of $\p_l^{9\delta}$ passing through $p_l$. Then the dist$(y_l, \{x \in B_4: x_n = -4\delta\})< 9\delta + \gamma + 4\delta = 13\delta + \gamma$. Note that $\p_l^{9\delta} \cap B_4 \subset \Omega^l$. Since distance function between $\p_l^{9\delta} \cap B_4$ and $\{x \in B_4: x_n = -4\delta\}$  is nonnegative harmonic, we can apply Harnack Inequality.
\begin{align*}~\label{harnack2} \max_{ y \in \p_l^{9\delta} \cap B_4} {\rm{dist}} &(y, \{x \in B_4: x_n = -4\delta\})  \\
& \leq C \min_{ y \in \p_l^{9\delta} \cap B_4}  {\rm{dist}}(y, \{x \in B_4 :x_n = -4\delta\})\\& \leq C {\rm{dist}}(y_l, \{x \in B_4: x_n = -4\delta\})\\&= C( 13\delta + \gamma)
\end{align*}
where $C$ depends on the dimension $n$. One can choose small $\gamma_0$ and $\delta_0$ so that $C( 13\delta_0 + \gamma_0) <\delta$ for $\delta$ in Lemma \ref{glolem}. We conclude by Lemma \ref{glolem}.

Now suppose dist$(\d \Omega^l, \d \Omega) \geq \gamma_0$ in $B_4$ for the $\gamma_0$ above. For any $ y \in S_1 = \{ x \in B_1 |\h x \in \d \Omega^l \h \} $, $B_{\gamma_0}(y)$ has two subdomains and $B_{\gamma_0}(y) \cap \d \Omega = \emptyset$. From (\ref{mynoemptyr}), there exists $x_0 \in \Omega_1$ such that
$$\M(|\g u |^2 )(x_0) \leq 1 \h\h\h \mbox{ and}\h\h\h \M (|f|^2)(x_0)\leq \delta^2.$$







\noindent
As we showed in the proof of Lemma \ref{bigducontr}, there is a constant $N_1$ so that for any $\vare > 0$, there exists $\delta > 0$ so that

\begin{equation}~\label{gl1} |\{ x \in \Omega  : \M (| \g u |^2)(x) > \lambda_1 N_1^2\} \cap B_{ \frac{\gamma_0}{4}}(y) | < \vare| B_{\frac{\gamma_0}{4}}(y)|.\end{equation}
where $\lambda_1 > \frac{C2^{3n+1}}{\gamma_0^n}$. Also for any $ y \in S_2 = \{ x \in B_1 |\h x \in \d \Omega \h \} $, $B_{\gamma_0}^+ \subset \Omega^0 \subset B_{\gamma_0} \cap \{ x_n > -\gamma_0\delta\}$ in appropriate coordinate system. By applying Corollary \ref{glocor}, there is a constant $N_1$ so that for any $\vare > 0$, there exists $\delta > 0$ so that

\begin{equation}~\label{gl2}|\{ x \in \Omega  : \M (| \g u |^2)(x) > \lambda_1 N_1^2\} \cap B_{ \frac{\gamma_0}{4}}(y) | < \vare| B_{\frac{\gamma_0}{4}}(y)|.\end{equation}

\noindent
For any $ y \in T = \{ x \in B_1 |\h \min ({\rm{dist}}(x, \d \Omega^l), {\rm{dist}}(x, \d \Omega))> \frac{\gamma_0}{4 \times 5} \h \}$, $ B_{\frac{\gamma_0}{20}}(y) \subset \Omega^i$ for $ i = 0, l $. Then by Lemma \ref{bigducontrol1} there is a constant $N_1$ so that for any $\vare > 0$, there exists $\delta > 0$ so that

\begin{equation}~\label{gl3}|\{ x \in \Omega  : \M (| \g u |^2)(x) > \lambda_2 N_1^2\} \cap B_{ \frac{\gamma_0}{20 \times 4}}(y) | < \vare| B_{\frac{\gamma_0}{80}}(y)|\end{equation}
where $\lambda_2 >  \frac{C2^{5n+1}5^n}{\gamma_0^n}$.

\noindent
Since $B \subset  U \defeq \{ B_r(y) | \h r < \frac{\gamma_0}{4 \times 5}, y \in S_1 \cup S_2 \h \} \cup  \{ B_r(y) | \h r < \frac{\gamma_0}{80 \times 5}, y \in T \h \} $, by Vitali Covering Lemma, there are disjoint set $\{ B_{r_i}(y_i) \}_{i = 1}^{\infty} \subset U \subset B_2$ s.t. $B_1 \subset  \cup_{i}B_{5r_i}(y_i)$
\begin{align*} |\{ x \in \Omega  :  & \M (| \g u |^2)(x) > N_1^2\} \cap B_1 | \\
& < \sum_{i}|\{ x \in \Omega  : \M (| \g u |^2)(x) > N_1^2\} \cap B_{5r_i}(y_i) |\\
& < \vare \sum_{i}  | B_{5r_i}(y_i)| < \vare 5^n\sum_{i}  | B_{r_i}(y_i)| \\
& < \vare 5^n |B_2| < \vare (10)^n |B_1|.
\end{align*}

\noindent
Here we used (\ref{gl1}), (\ref{gl2}) and (\ref{gl3}).
\end{proof}

\begin{cor}~\label{myglocor}There is a constant $ N_1 >0 $ so that for any $\vare> 0$, there exists a small $ \delta= \delta(\vare) > 0$ with $A_i$'s are uniformly elliptic and $(\delta, 9)$-vanishing on $\Omega^i$ for $i = 0 \dots K$ and $\Omega$ and $\Omega^i$'s are $(\delta,9)$-flat for $i = 1\dots K$, and if $ u  \in H^1_0(\Omega)$ is a weak solution of (\ref{divglou}) with $B_{4r}^+ \subset  \Omega_{4r} \subset B_{4r} \cap \{ x_n > -4\delta r \} $ and
\begin{equation}~\label{mynoemptyr1}  \{ x \in \Omega_r : \M(|\g u |^2 ) \leq 1 \} \cap \{x \in \Omega_r: \M (|f|^2)\leq \delta^2 \} \neq \emptyset, \end{equation}
then
\begin{equation}~\label{mymeasuresmallr1}| \{ x \in \Omega : \M (| \g u |^2)(x) > N_1^2\} \cap B_r| < \vare|B_r|.\end{equation}
\end{cor}

\begin{proof} Then proof is given by Lemma \ref{myglolem} and scaling argument.
\end{proof}

The following lemma shows that same result of Lemma \ref{myglolem} holds for any ball intersecting with $\Omega$. 

\begin{lem}~\label{tousevitali1} There is a constant $ N_1 >0 $ so that for any $ \vare > 0 $ and $0<r<1$, there exists a small $ \delta= \delta(\vare) > 0$ for all $\Omega = \cup_{i=0}^{K} \Omega^i$ where $ \Omega^0 \defeq {\Omega\setminus\cup_{i=1}^K\Omega^i} $ and $\Omega$ and disjoint subdomains $\Omega^i$'s for $i = 1, \dots, K$ are $(\delta, 45)$-flat and for any  $A=\sum_{i=0}^{i=K}A^i\chi_{\Omega^i}$ where $A_i$'s are uniformly elliptic and $(\delta, 45)$-vanishing on $\Omega^i$, and if $u \in H_0^1(\Omega)$ is the weak solution of $ -{\rm{div}}(A \g u )= {\rm{div}} f $ in $ \Omega \supset B_{4r}$ and if the following property holds:
\begin{equation}~\label{tousevitali2} \{ x \in \Omega_r : \M(|\g u |^2 ) \leq 1 \} \cap \{x \in \Omega_r: \M (|f|^2)\leq \delta^2 \} \neq \emptyset,\end{equation}
then
\begin{equation}\label{tousevitali3}| \{ x \in \Omega : \M (| \g u |^2)(x) > N_1^2\} \cap B_r| < \vare|B_r|.\end{equation}
\end{lem}

\begin{proof}
If $B_{4r} \cap \d \Omega = \emptyset$, then by an interior estimate \ref{110main} we can conclude. Assume that $B_{4r} \cap \d \Omega \neq \emptyset$. Note that $B_r \subset B_{5r}(y)$ for some $y \in \d \Omega$. By (\ref{tousevitali2}), there exists $x_0 \in B_r \subset B_{5r}(y)$ such that $\M(|\g u |^2 )(x_0) \leq 1$ and $\M (|f|^2)(x_0)\leq \delta^2$. Since $\Omega$ is ($\delta, 45$)-Reifenberg flat, we have, in appropriate coordinate system,
$$B_{20r}^+ \subset  \Omega_{20r} \subset B_{20r} \cap \{ x_n > -20\delta r \}. $$
Here we use the corollary \ref{myglocor} to the ball $B_{5r}(y)$ with $\vare$ replaced by $\frac{\vare}{5^n}$. Then
\begin{align*}| \{ x \in \Omega : \M (| \g u |^2)(x) > N_1^2\} \cap B_r| & < | \{ x \in \Omega : \M (| \g u |^2)(x) > N_1^2\} \cap B_{5r}(y)| \\
& <\frac{\vare}{5^n} |B_{5r}| = \vare |B_r|.
\end{align*}
\end{proof}

\begin{cor}\label{induction} (cf.\cite{BW03}) Suppose that $ u \in H_0^1(\Omega) $ is the weak solution of $-{\rm{div}}(A\g u) = {\rm{div}} f $ in $ \Omega$.  Assume $\Omega = \cup_{i=0}^{K} \Omega^i$ where $ \Omega^0 \defeq {\Omega\setminus\cup_{i=1}^K\Omega^i} $ and  $\Omega^i$'s for $i = 1, \dots, K$ and $\Omega$ are $(\delta, 45)$-flat and $A=\sum_{i=0}^{i=K}A^i\chi_{\Omega^i}$ where $A_i$'s are uniformly elliptic and $(\delta, 45)$-vanishing in $\Omega^i$. Assume that
\begin{equation}~\label{induction1} | \{ x \in \Omega : \M (|\g u |^2) > N_1^2\} | < \vare |B_1|. \end{equation}
Let $k$ be a positive integer and set $\vare_1 = (\frac{10}{1-\delta})^n\vare.$ Then we have
\begin{align}~\label{inductiomn3}| \{ & x \in \Omega :  \M (|\g u |^2) > N_1^{2k}\} | \\ & \leq \sum_{i = 1}^{k} \vare_1^i|\{ x \in \Omega : \M(|f|^2) > \delta^2 N_1^{2(k-i)}| + \vare_i^k |\{ x\in \Omega : \M (|\g u | ) ^2 (x) > 1 \} |.\end{align}
\end{cor}

\begin{proof} We prove by induction on $k$.  For the case $ k = 1$, set
\begin{equation*}C =\{ x \in \Omega : \M  (|\g u |^2)(x) > N_1^{2} \}\end{equation*}
and
\begin{equation*}D = \{ x \in \Omega : \M  (|f |^2)(x) > \delta^{2} \} \cup \{ x \in \Omega : \M  (|\g u |^2)(x) > 1 \}. \end{equation*}

\noindent
Since $ \Omega$ is $ (\delta, 45) $-Reifenberg flat, $ \Omega $ is $(\delta, 1)$-Reifenberg flat. Then in view of (\ref{induction1}), Lemma \ref{tousevitali1} and Theorem \ref{modifiedvitali}, we see $ |C| \leq \vare_1|D|$, and so our conclusion is valid for $ k=1$.

Assume that the conclusion is valid for some positive integer $ k \geq 2. $ Set $u_1 = u / N_1$ and corresponding $f_1 = f / N_1.$ Then $u_1$ is the weak solution of
\begin{equation}~\label{div1}
\begin{cases} -\mbox{div}( A\g u_1) =  \mbox{div} f_1   \h\h \mbox{in}\h \Omega  \\ \hspace{1.5cm}u_1 = 0 \h\h\h \mbox{on} \h\d \Omega
\end{cases}\end{equation}

\noindent
and the following inequality holds:
\begin{equation*}|\{ x \in \Omega : \M  (|\g u_1 |^2)(x) > N_1^{2}\}| < \vare |B_1 |. \end{equation*}

By the induction assumption and from a simple calculation, we deduce the following estimates:
\begin{align*}|\{ x & \in \Omega : \M  (|\g u |^2)(x) > N_1^{2(k+1)}\}| \\
& = |\{ x \in \Omega : \M  (|\g u_1 |^2)(x) > N_1^{2k}\}| \\
& \leq \sum_{i = 1}^{k} \vare_1^i |\{ x \in \Omega : \M  (|f_1 |^2)(x) > \delta^2 N_1^{2(k-i)}\}| \\
& \h\h + \vare_1^k |\{ x \in \Omega : \M  (|\g u_1 |^2)(x) > 1\}| \\
& \leq \sum_{i = 1}^{k+1} \vare_1^i |\{ x \in \Omega : \M  (|f |^2)(x) > \delta^2 N_1^{2(k+1-i)}\}| \\
& \h\h + \vare_1^{k+1} |\{ x \in \Omega : \M  (|\g u |^2)(x) > 1\}|.
\end{align*}

\noindent
This estimate in turn completes the induction on $k$.
\end{proof}

Finally we are ready to prove the main theorem.
\begin{thm}Let $p$ be a real number $ 1 < p < \infty.$ Then there is a small $\delta= \delta(\Lambda, p , n, R)> 0 $ so that for all $\Omega = \cup_{i=0}^{i=K} \Omega^i$ where $\Omega^0 \defeq {\Omega\setminus\cup_{i=1}^{i=k}\Omega^i}$ and $\Omega$ and disjoint subdomains $\Omega^i$'s for $i=1,\dots,K$ are $(\delta, R)$-Reifenberg flat, for all  $A=\sum_{i=0}^{i=K}A^i\chi_{\Omega^i}$ where $A^i$'s are  $(\delta, R)$-vanishing in $ \Omega^i$ and  uniformly elliptic for $i=0,\dots,K$, and for all $f$ with $f \in L^p(\Omega, \bbr^n),$ the Dirichlet problem (\ref{div}) has a unique weak solution with the estimate
\begin{equation} \displaystyle \int_{\Omega}|\g u |^p dx \leq C \displaystyle \int_{\Omega}|f|^p dx,\end{equation}
where the constant $C$ is independent of $u$ and $f$.
\end{thm}

\begin{proof}First we will consider the case $ p> 2$. The case $ p =2 $ is classical and the case $ 1 < p < 2$ will be proved using duality. Without loss of generality, we assume that
\begin{equation}~\label{flpsmall} \|f\|_{L^p(\Omega)} \mbox{ is small enough}\end{equation}
and $$ | \{ x \in \Omega : \M (|\g u |^2) > N_1^2\} | < \vare |B_1|$$

\noindent
by multiplying the PDE (\ref{div}) by a small constant depending on $\|f\|_{L^2(\Omega)} $ and $\| \g u\|_{L^2(\Omega)}$. Since $ f \in L^p(\Omega), \M(|f|^2) \in L^{p/2}(\Omega) $ by strong p-p estimates. In view of Lemma \ref{lplevel}, there is a constant $C$ depending only on $\delta, p,$ and $N_1$ such that

\begin{equation} \sum_{k=0}^{\infty} N_1^{pk} | \{ x \in \Omega : \M (| f |^2)(x) > \delta^2 N_1^{2k}\} \leq C\|\M(|f|^2)\|_{L^{p/2}(\Omega)}^{p/2}. \end{equation}
Then this esitmate, strong p-p estimates, and (\ref{flpsmall}) imply
\begin{equation}~\label{flpless} \sum_{k=0}^{\infty} N_1^{pk} | \{ x \in \Omega : \M (| f |^2)(x) > \delta^2 N_1^{2k}\} \leq 1. \end{equation}

Now we will claim that $ \M(|\g u |^2) \in L^{p/2}$ by using Lemma \ref{lplevel} when $f = \M(|\g u |^2) $ and $ m = N_1^2.$ Let us compute
\begin{align*}
&\sum_{k=0}^{\infty}  N_1^{pk} | \{ x \in \Omega : \M (| \g u  |^2)(x) >  N_1^{2k}\}| \\
&\leq \sum_{k=1}^{\infty} N_1^{pk} \left( \sum_{i = 1}^{k} \vare_1^i| \{ x \in \Omega : \M (| f |^2)(x) > \delta^2 N_1^{2(k-i)}\}| + \vare_1^k  | \{ x \in \Omega : \M (| \g u |^2)(x) > 1\} | \right)\\
& =  \sum_{i=1}^{\infty} (N_1^{p}\vare_1)^i \left( \sum_{k = i}^{\infty} N_1^{p(k-i)}| \{ x \in \Omega : \M (| f |^2)(x) > \delta^2 N_1^{2(k-i)}\}| \right)\\
&\hspace{1.7cm} + \sum_{k = 1}^{\infty} (N_1^p \vare_1)^k  | \{ x \in \Omega : \M (| \g u |^2)(x) > 1\} |\\
& \leq  C \sum_{k = 1}^{\infty} (N_1^p \vare_1)^k < + \infty,
\end{align*}
where we used Corollary \ref{induction} and (\ref{flpless}). Also we can choose $ \vare_1$ so that $N_1^p\vare_1 < 1 $ since $N_1 $ is a universal constant depending on the dimension and ellipticity. So we can take $ \vare$, and find the corresponding $\delta > 0,$ also $\vare_1$. By this estimate and Lemma \ref{lplevel}, $\M(|\g u |^2) \in L^{p/2}(\Omega)$. Thus $\g u  \in L^{p}(\Omega).$

Now suppose that $1<p<2$. For any $g \in L^q(\Omega, \bbr^n)$ and $A^T$, a transpose matrix of $A$, consider the following equation.
\begin{equation}\label{divdual}
\begin{cases} -\mbox{div}( A^T (x)\g v(x)) =  \mbox{div}  g \hspace{1cm} \mbox{in}\h \Omega  \\ \hspace{1.7cm} v = 0 \h\h\h\h\h\h\h\h\mbox{on} \h \d \Omega
\end{cases}\end{equation}

\noindent
Then 
\begin{align*}
\displaystyle\int_{\Omega} f \g v dx  & = - \displaystyle\int_{\Omega}  \mbox{div} f v dx = \displaystyle\int_{\Omega}  \mbox{div} (A \g u ) v dx \\
&=-\displaystyle\int_{\Omega} (A\g u)(\g v) dx = -\displaystyle\int_{\Omega}  \g u (A^T \g v) dx \\ 
&= \displaystyle\int_{\Omega} u \mbox{div}(A^T \g v) dx =\displaystyle\int_{\Omega} u (- \mbox{div} g) dx = \displaystyle\int_{\Omega} \g u g dx. 
\end{align*}

\noindent
By above, note that $\| \g v \|_{L^q} \leq C \|g\|_{L^q},$
\begin{align*}
\| \g u \|_{L^p(\Omega)} & = \sup_{0\neq g \in L^q(\Omega)} \frac{|\displaystyle \int_{\Omega}\g u g |}{\|g\|_{L^q(\Omega)}} \leq \frac{|\displaystyle \int_{\Omega}\g v f |}{\|g\|_{L^q(\Omega)}}\\
&\leq \frac{\|\g v \|_{L^q}\|f\|_{L^p}}{\| g\|_{L^q}} \leq C\|f\|_{L^p},
\end{align*}
\noindent
which completes the proof.
\end{proof}


\end{document}